\documentclass{amsart}
\usepackage[dvipsnames]{xcolor}
\usepackage{hyperref}
\hypersetup{
     colorlinks=true,
     linkcolor=blue!60!black,
     filecolor=blue!75!black,
     citecolor = green!50!black,      
     urlcolor=magenta,
     }
\usepackage{ stmaryrd }
\usepackage{amssymb}
\usepackage{ bbold }
\usepackage[all]{xy}
\xyoption{all}
\usepackage{url}
\usepackage{mathrsfs}
\usepackage{tikz-cd}
\usepackage{eucal}
\usepackage{comment}
\usepackage{mathtools}

\usepackage[nameinlink,capitalise,noabbrev]{cleveref}

\newcommand{\Hom}{\operatorname{Hom}\nolimits}

\def \D{{\mathcal D}}
\def \T{{\mathcal T}}
\def \S{{\mathcal S}}

\DeclareMathOperator{\h}{h}

\DeclareMathOperator{\Mod}{Mod}
\DeclareMathOperator{\mmod}{mod}

\usepackage{todonotes}

\newcommand{\mmmod}{\mmod\kern-0.1em\text{-}}%     
\newcommand{\MMod}{\Mod\kern-0.1em\text{-}}%  

\numberwithin{equation}{section}% makes equat numb contain the section
\newtheorem{theorem}[equation]{Theorem}
\newtheorem{proposition}[equation]{Proposition}
\newtheorem{corollary}[equation]{Corollary}
\newtheorem{lemma}[equation]{Lemma}

\theoremstyle{remark}
\newtheorem{definition}[equation]{Definition}

\newtheorem{remark}[equation]{Remark}
\newtheorem{recollection}[equation]{Recollection}
\newtheorem{example}[equation]{Example}

\newtheorem{notation}[equation]{Notation}

\newtheorem{convention}[equation]{Convention}

\newtheorem{warning}[equation]{Warning}
\newtheorem{question}[equation]{Question}

\thanks{}
\title{Conservative geometric functors via purity}

\begin{document}

\author[]{Natàlia Castellana}
\author[]{Juan Omar G\'omez}

\address{Natàlia Castellana, Departament de Matemàtiques, Universitat Autònoma de Barcelona, 08193 Bellaterra, Spain; Centre de Recerca Matemàtica, Barcelona, Spain}
\email{Natalia.Castellana@uab.cat}
\urladdr{https://mat.uab.cat/~natalia}

\address{Juan Omar G\'omez, Fakultat f\"ur Mathematik, Universit\"at Bielefeld, D-33501 Bielefeld, Germany}
\email{jgomez@math.uni-bielefeld.de}
\urladdr{https://sites.google.com/cimat.mx/juanomargomez/home}

\begin{abstract}
 We establish a criterion for determining when a family of geometric functors is jointly conservative through the lens of purity in compactly generated triangulated categories. We introduce the notion of pure descendability and we apply it to two particular situations involving sequential limits of ring spectra.  
\end{abstract}

\maketitle

%\tableofcontents

\setcounter{tocdepth}{1}
\setlength{\parindent}{0cm}
\setlength{\parskip}{0.8ex}

\section{Introduction}

A tensor-triangulated (tt) functor $f^\ast\colon \T\to \S$ between rigidly-compactly generated tt-categories is called \textit{geometric} if it preserves small coproducts. Such functors serve as auxiliary tools for understanding geometric properties of $\T$ via those of $\S$. In fact, geometric functors already possess notable structural features: they admit a chain of adjunctions $f^\ast\dashv f_\ast \dashv f^!$, and the pair $(f^\ast,f_\ast)$ satisfies the projection formula; see \cite{BDS16}. 

 A particularly desired property of a geometric functor is \textit{conservativity}. This property has strong consequences on Balmer spectra \cite{balmer2018surjectivity}.  In practice, however, one often encounters a family of geometric functors $(f^\ast_i\colon \T\to \S_i)_{i\in \mathcal{I}}$ where $\mathcal{I}$ is not necessarily finite, and one instead asks for \textit{joint conservativity} of the family which is still sufficient to descend substantial geometric information; see for instance \cite{BCHS} and \cite{BCHS24}.

Examples of jointly conservative families of geometric functors include the following:

 \begin{enumerate}
     \item Let $G$ be a finite group and $k$ be a field of positive characteristic $p$ dividing the order of $G$. Then the family of restriction functors between stable module categories
     \[
      \left(\mathrm{Res}^G_E \colon \mathrm{StMod}_{kG} \to \mathrm{StMod}_{kE} \right)_{E\in \mathcal{E}_p(G)}
     \]
     is jointly conservative, where $\mathcal{E}_p(G)$ is the collection of elementary abelian $p$-subgroups of $G$. In fact, this is a reformulation of a classical theorem of Chouinard \cite{Chouinard} in modular representation theory.
     
     \item Let $G$ be a compact Lie group and $k$ be a field of positive characteristic $p$.  Write $\Mod_{C^\ast(BG;k)}$ to denote the category of module spectra over the algebra of $k$-valued cochains on the classifying space ${C^\ast(BG;k)}$.  Then, the family of induction functors 
     \[
      \left(\mathrm{Ind}^E_G \colon \Mod_{C^\ast(BG;k)} \to \Mod_{C^\ast(BE;k)} \right)_{E\in \mathcal{E}_p(G)}
     \]
     is jointly conservative. As before, $\mathcal{E}_p(G)$ denotes the collection of elementary abelian $p$-subgroups of $G$ up to conjugacy. See \cite{BG}.
     
     \item Let $R$ be a commutative Noetherian ring. Then the family of derived base change functors  along residue fields  
     \[
     \left(-\otimes^L_Rk(\mathfrak p) \colon \mathbf{D}(R) \to \mathbf{D}(k(\mathfrak p)\right)_{\mathfrak p\in \mathrm{Spec}(R)}
     \]
     is jointly conservative; see for instance \cite{Nee92}. 
 \end{enumerate}

 We emphasize that in the previous examples the relevance of these families extends beyond tt-geometry: they also provide a more refined understanding of a category in terms of simpler ones. This naturally leads to the following question:

 \begin{question}\label{question-main}
 Under what conditions on a family of geometric functors $(f_i^\ast)_{i\in \mathcal I}$ can we ensure that they form a jointly conservative family? 
 \end{question}

This question has been studied in the context of \textit{descent in tt-geometry}. In particular,  weak descendability provides sufficient conditions for the family $(f_i^\ast)_{i\in\mathcal I}$ to be jointly conservative; see for instance \cite{BCHS}, \cite{barthel2024homological} and \cite{Gom25}. Informally, weak descendability requires that the monoidal unit be reconstructed from the objects $(f_i)_\ast (\mathbb 1_{\mathcal S_i})$ as a localizing tensor ideal; see \cref{rec-weakly} for the precise definition. In fact, as explained in \textit{loc. cit.}, weak descendability has even stronger consequences. 

In this paper, we address \cref{question-main} through the lens of purity in compactly generated triangulated categories, as developed by Krause  in \cite{Kra00}. This notion was also introduced independently by Beligiannis in \cite{Beligiannis}. In particular, we introduce the notion of a \textit{pure descendable} family of geometric functors; see \cref{def-puredesc} for further details.

%%%%%%%%%%%%%%%%------------------------------------------------------------------------------------------------

\begin{definition}
    Let $(f_i^\ast\colon \T\to \S_i)_{i\in \mathcal I}$ be a family of geometric functors. For a collection of objects $\mathcal{X}$, we write  $\mathbf{pure}^\Delta_\Pi(\mathcal{X})$ as short for the pure-closure of the smallest thick subcategory of $\T$ containing the product-closure  of  $\mathcal{X}$.   We say that the family $(f_i^\ast)_{i\in \mathcal{I}}$ is \textit{pure descendable} if 
 $\mathbf{pure}^\Delta_\Pi\left(\bigcup_{i\in \mathcal I}(f_i)_\ast\S_i\right)$ is a tensor-ideal and  contains the monoidal unit.
\end{definition}

Our motivation for introducing \textit{pure descendability} is to formulate a more flexible version of \textit{weak descendability}, one that still guarantees joint conservativity but also allows for arbitrary products. The key difference is that pure descendability dispenses with closure under arbitrary fibers and cofibers, while permitting products.

In fact, our choice of terminology is motivated by weak descendability, and is meant to highlight that we seek to reconstruct the monoidal unit using the purity in the category $\T$, rather than relying solely on its triangulated and monoidal structures.

Some of the motivation for this definition arises from contexts with a rich homological structure, which we aim to exploit in order to produce examples of pure descendable families.  Before elaborating this further, let us record the main consequence of this definition, which appears later as \cref{prop:pure-desendable}.   

\begin{theorem}
     If $(f_i^\ast\colon \T\to \S_i)_{i\in \mathcal I}$ is a pure descendable family of geometric functors, then it is jointly conservative.  
\end{theorem}

With this result in place, one might now wonder how to verify the conditions appearing in the definition of a pure descendable family 
$(f_i^\ast)_{i\in \mathcal{I}}$.  We next provide a criterion for determining when the monoidal unit lies in $\mathbf{pure}^\Delta_\Pi\left(\bigcup_{i\in \mathcal I}(f_i)_\ast\S_i\right)$ in the setting of module categories over sequential limits of ring spectra. This criterion will play a crucial role in establishing our main examples of pure descendable families. The following result corresponds to \cref{coro-puremono-seqlimit}.

\begin{theorem}\label{thm-seqlim}
       Let $k$ be a commutative Artinian ring. Let $R$ be a commutative ring spectrum such that $\pi_\ast R$ is a $k$-algebra. Suppose that 
   \[
   R\simeq \lim_{n\in \mathbb N } R_n
   \]
   in $\mathrm{CAlg}$, with $\pi_\ast R_n$ a finitely generated $k$-algebra. Then the canonical map
   \[
   R\to \prod_{n\in \mathbb N} R_n 
   \]
   is a pure monomorphism in $\Mod_R$. 
\end{theorem} 

Let us make two remarks regarding the previous result. First, one can generalize this statement to more general rigidly-compactly generated tt-categories; see \cref{prop:pure-mono in Artinian}. Second, there is work dedicated to determining when a limit of pure-injective modules remains pure-injective. In such cases, one may conclude that the ring $R$ itself is a pure-injective object, which in turn implies that $R$ is a retract of $\prod_n R_n$. This can be used, for example, to recover the fact that the ring of $p$-adic integers is a retract of $\prod_{n\in \mathbb N}  \mathbb{Z}/p^n\mathbb Z$; see  \cref{ex-completelocal}.

The original motivation for this project was to establish a Chouinard-type theorem for the category of module spectra over the cochain algebra $C^\ast(BG;\mathbb F_p)$, where $G$ is a locally finite Artinian $p$-group (see \cite[Section 4.3]{BCHV}).

In fact, for a locally finite Artinian group $G$, the ring $C^\ast(BG;\mathbb F_p)$ falls into the framework of \cref{thm-seqlim}. This is a key ingredient in establishing our main example of a pure descendable family of geometric functors, stated in \cref{thm-example of pure-des}.

\begin{theorem}
    Let $G$ be locally finite Artinian group expressed as $G=\bigcup_{n\in \mathbb{N}} G_n$ for an ascending chain of finite subgroups $\cdots \subset G_n\subset G_{n+1}\subset \cdots$ of $G$. Then the family of geometric functors 
    \[
    \left(\mathrm{Ind}_G^{G_n}\colon \Mod_{C^\ast(BG;\mathbb F_p)} \to \Mod_{C^\ast(BG_n;\mathbb F_p)} \right)_{n\in \mathbb N}
    \]
    is pure descendable.
\end{theorem} 

As a consequence, we obtain the desired Chouinard-type theorem for locally finite Artinian groups; see \cref{coro: Chouinard for discrete p-toral}.

\begin{corollary}
     Let $G$ be locally finite Artinian group.   Let  
 $\mathcal{E}$ be a set of representatives of $G$-conjugacy classes of elementary abelian $p$-subgroups of $G$. Then  the geometric functor 
    \[
    \mathrm{Ind}_{\mathcal{E}}\colon \Mod_{C^*(BG;\mathbb{F}_p)}\rightarrow \prod_{E\in\mathcal{E}} \Mod_{C^*(BE;\mathbb{F}_p)}
    \]
    induced by the functors $\mathrm{Ind}_G^E\colon \Mod_{C^*(BG;
    \mathbb{F}_p)}\rightarrow \Mod_{C^*(BE;\mathbb{F}_p)}$, is conservative. 
\end{corollary}

This result brings us closer to completing the (co)stratification program for homotopical groups studied in \cite{BCHV}, with the conservativity of the analogous coinduction functor being the only remaining obstruction to such a conclusion. An interesting open question is whether the methods developed in this paper can also be applied to address this problem.

The paper is organized as follows. Section 2 contains preliminaries on tensor-triangulated geometry. In Section 3 we discuss purity in tensor-triangulated categories. In Section 4 we introduce the notion of pure descendable families of functors. Sections 5 and 6 are devoted to applications: sequential limits of ring spectra and geometric cochains of locally finite Artinian groups.

\subsection*{Notation and Conventions} 
Throughout this work, we freely use notation and terminology from \cite{BDS16} and \cite{BCHV}, to which we refer for any unexplained concepts. We often assume that our triangulated categories arise from stable $\infty$-categories. When the context is clear, we use `tt-category' as shorthand for  `symmetric monoidal stable $\infty$-category'. 

Now, let  $\T$ be a rigidly-compactly generated symmetric monoidal stable $\infty$-category. Then we write: 
\begin{enumerate}
    \item $\Hom_\T(-,-)$ to denote the mapping space in $\T$; 
     \item $\Hom^\ast_\T(-,-)=\Hom_\T(\Sigma^*-,-)$ to denote the graded homomorphisms in the homotopy category of $\T$;
    \item $\T(-,-)$ the group of homomorphisms in the homotopy category of $\T$.
\end{enumerate}

\subsection*{Acknowledgements}
We are deeply grateful to Isaac Bird for answering numerous questions about purity. We also thank Sergio Pavon and Mike Prest for their helpful comments, which have led us to improve our definition of \textit{pure descendability}, Wassilij Gnedin for providing relevant references, and Drew Heard for comments on an early version of this work. 
We are grateful to the anonymous referee for a careful reading of the manuscript and for numerous helpful comments and suggestions.

NC would like to thank the Isaac Newton Institute for Mathematical Sciences, Cambridge, for support and hospitality during the programme Equivariant homotopy theory in context where work on this paper was undertaken. This work was supported by EPSRC grant no EP/R014604/1. NC is partially supported by Spanish State Research Agency project PID2024-158573NB-I00, the Severo Ochoa and María de Maeztu Program for Centers and Units of Excellence in R$\&$D (CEX2020-001084-M), and the CERCA Programme/Generalitat de Catalunya. JOG is supported by the Deutsche Forschungsgemeinschaft (Project-ID 491392403 – TRR 358).

\section{Preliminaries}

In this section, we recall the relevant terminology to be used throughout the paper. We refer the reader to \cite{BDS16} and \cite{Ste18} for further details.

\begin{definition}\label{def:rcg tt cat}
    A \textit{rigidly-compactly generated} tensor-triangulated category $\T$ is a triple $(\T,\otimes, \mathbb 1)$ where $\T$ is compactly generated and $(\otimes, \mathbb 1)$ is a symmetric monoidal structure on $\T$ satisfying that $\otimes$ is coproduct preserving in each variable, the monoidal unit $\mathbb 1$ is a compact object, and any compact object is dualizable.  
\end{definition}

\begin{remark}
Note that the assumptions in \cref{def:rcg tt cat} imply that an object is compact if and only if it is dualizable.  
\end{remark}

\begin{remark}
    A rigidly-compactly generated tt-category is automatically closed monoidal.
\end{remark}

\begin{remark}
   In the literature, rigidly-compactly generated tt-categories are also referred to as \textit{big} tt-categories. We do not adopt this terminology in the present paper. 
\end{remark}

\begin{definition}
    A strongly monoidal triangulated functor $f^\ast\colon \T\to \S$ between rigidly-compactly generated tt-categories is a \textit{geometric functor} if it is coproduct preserving. 
\end{definition}

As mentioned in the introduction, geometric functors enjoy notable structural properties.

\begin{recollection}
    Let $f^\ast \colon \T\to \S$ a geometric functor. Then there is a chain of adjunctions 
    \[
    f^\ast\dashv f_\ast \dashv f^!
    \]
    and the pair $(f^\ast,f_\ast)$ satisfies the \textit{projection formula}: for any $x$ in $\T$ and any $y$ in $\S$ there is natural isomorphism 
    \[
    f_\ast (f^\ast (x)\otimes y)\simeq x\otimes f_\ast (y). 
    \]
\end{recollection}

\begin{definition}
    Let $\mathcal{I}$ be a set and let $(f_i\colon \T\to \S_i)_{i\in \mathcal{I}}$ be a family of triangulated functors between arbitrary triangulated categories. We say that the family $(f_i)_{i\in \mathcal{I}}$ is \textit{jointly conservative} if, for any object $x\in \T$, one has $x\simeq 0$ if and only if $f_i(x)\simeq 0$ for all $i\in \mathcal{I}$.
\end{definition}

\begin{remark}
   In general, a functor is said to be \textit{conservative} if it reflects isomorphisms. In the triangulated setting, this notion agrees with the above definition of conservativity. 
\end{remark}

\section{Purity in triangulated categories}

In this section, we recollect some terminology and basic results on definable categories and purity in compactly generated triangulated categories that will be used later in this document. Our main references are \cite{BW23} and \cite{Kra00}.  

\begin{recollection}\label{basic res of purity}
    Let $\T$ be a compactly generated triangulated category. We write $\T^c$ to denote the full subcategory of $\T$ on compact objects. Let $\MMod\T^c$ denote the abelian category of (right) $\T^c$--modules, that is, the category of additive functors $(\T^c)^\mathrm{op}\to \mathrm{Ab}$. We also consider the restricted Yoneda functor  
    \[\h\colon \T\to \MMod\T^c, \quad x\mapsto {\T}(-,x)|_{\T^c}.\] 
    The restricted Yoneda functor sends distinguished triangles to long exact sequences of abelian groups. 

    A map $f\colon x\to y$ in $\T$ is a \textit{pure monomorphism} if $\h(f)$ is a monomorphism, that is, $\h(f)(c)$ is injective for all $c\in \T^c$. Dually, we define \textit{pure epimorphisms}. An object $x\in \mathcal{T}$ is a \textit{pure subobject of} $y$ if there is a pure monomorphism $x\to y$. An object $x\in \T$ is \textit{pure-injective} if $\h(x)$ is injective in $\MMod\T^c$. Equivalently, $x$ is pure-injective if any pure monomorphism $x\to y$ splits. A morphism $f\colon x\to y$ is \textit{phantom} if $\h(f)=0$.  A \textit{pure triangle} is a distinguished triangle 
    \[
    x\xrightarrow[]{}x'\xrightarrow[]{}x''\xrightarrow[]{}\Sigma x \quad  \mbox{such that} \quad 0\to \h(x)\to \h(x')\to \h(x'')\to 0
    \]
     is exact.
\end{recollection}

\begin{recollection}\label{rec-phantoms-pure monos}
     There is a strong connection between pure monomorphisms, pure epimorphisms and phantom maps. Namely, given a distinguished triangle 
     \[
     x\xrightarrow[]{f}x'\xrightarrow[]{f'}x''\xrightarrow[]{f''}\Sigma x,
     \]
     we have that $f$ is phantom if and only if $f'$ is a pure monomorphism if and only if $f''$ is a pure epimorphism.  See for instance \cite[Section 1]{Kra00}.
\end{recollection}

\begin{remark}
    Note that any pure monomorphism $f \colon x \to y$ in $\T$ extends to a pure triangle, namely 
    \[
    x\to y\to \mathrm{cone}(f). 
    \]
    Conversely, if 
    \[
    x\xrightarrow[]{f}x'\xrightarrow[]{f'}x''\xrightarrow[]{f''}\Sigma x
    \]
    is a pure triangle, then $f$ is a pure monomorphism, $f'$ is a pure epimorphism, and $f''$ is a phantom map.  
\end{remark}

 Our interest in pure-injective objects is the following detection property which is  Corollary 1.10 in \cite{Kra00}:

\begin{lemma}\label{lemma-krause 1.10}
 Let $\T$ be a compactly generated triangulated category.  Let $x \in \mathcal{T}$ such that $\mathcal{T}(x,y)=0$ for all indecomposable pure-injective objects $y\in \mathcal{T}$. Then $x=0$.    
\end{lemma}

We also need to recall the notion of a definable subcategory.

\begin{recollection}
   Let $\mathcal{T}$ be a compactly generated triangulated category. A functor $F\colon \mathcal{T}\to \mathrm{Ab}$ is \textit{coherent} if it has a presentation of the form \[\T(x,-)\xrightarrow[]{\T(f,-)} \T(y,-)\to F\to 0 \] for some $f\colon x\to y$ in $\mathcal{T}^c$. A full subcategory $\mathcal{D}$ of $\mathcal{T}$ is \textit{definable} if there is a set of coherent functors $\Phi$ such that \[\mathcal{D}=\{x\in \mathcal{T}\mid F(x)=0, \,  \forall F\in \Phi\}.\]
   For a given a class of objects $\mathcal{X}$ of $\mathcal{T}$, we write $\mathbf{Def}(\mathcal{X})$ to denote the smallest definable subcategory of $\mathcal{T}$ containing $\mathcal{X}$, and refer to it as the \textit{definable closure of $\mathcal{X}$}.  Since the intersection of definable subcategories is again definable, the previous closure operator is well defined. See \cite[Section 2]{BW23} for further details. 
\end{recollection}

\begin{remark}\label{rem-definible}
Let us emphasize that definable subcategories are not necessarily triangulated. Nevertheless, they are always closed under pure subobjects, pure quotients, and products (see \cite[(2.6)]{BW23}). Since any coproduct is a pure subobject of the corresponding product, definable subcategories are also closed under coproducts. Moreover, when the ambient category admits a suitable enhancement (e.g., under  \cref{convention}), definable subcategories can be characterized precisely as those closed under pure subobjects, pure quotients, and products (see \cite[(2.17)]{BW23}).
\end{remark}

\begin{convention}\label{convention}
    From now on, all triangulated categories, as well as exact functors between them, are assumed to arise from stable $\infty$-categories, without further mention.
\end{convention}

 With this convention in place, we can introduce the notion of a definable functor between compactly generated triangulated categories, following \cite[Definition 4.1]{BW23} and the equivalent characterizations of \cite[Theorem 4.17]{BW23}.

\begin{definition}\label{def:definable functor}
    Let  $f\colon \T\to \S$ be an exact functor between compactly generated triangulated categories. Then we say that $f$ is a \textit{definable functor} if it commutes with products and coproducts.
\end{definition}

We focus on definable functors because they interact well with definable subcategories. In particular, definable functors preserve pure triangles and pure-injective objects. A fundamental example to keep in mind is the following:

\begin{example}
    Let $f^\ast\colon \T\to \S$ be a geometric functor between rigidly-compactly generated tt-categories. Then the right adjoint $f_\ast\colon \S\to \T$ is a definable functor (see \cite[Example 5.3]{BW23}).  A criterion for when $f^*$ is definable is also given in \cite[Example 5.4]{BW23}, namely, when Grothendieck-Neeman duality holds for $f^\ast$. 
    In particular, this holds when the right adjoint $f^{!}$ of $f_*$ coincides with $f^*$. 
\end{example}

\begin{notation}
    Let $\mathcal{X}$ be a class of objects in a compactly generated triangulated category $\T$.  Following \cite{BW23}, we write $\mathbf{pure}(\mathcal{X})$ to denote the \textit{pure-closure of $\mathcal{X}$}, that is, the closure of $\mathcal{X}$ under pure subobjects.
\end{notation}

\begin{remark}
    Let $f^\ast\colon \mathcal{T}\to \mathcal{S}$ be a geometric functor between rigidly-compactly generated tt-categories, and let $\mathcal{D}$ be a definable subcategory of $\mathcal{S}$. If $x$ is a pure-injective object in $\mathcal{T}$ which lies in $\mathbf{pure}(f_\ast\mathcal{D})$, then there is a pure monomorphism $x\to f_\ast(y)$ for some $y\in \D$. Using that $x$ is pure-injective, such pure monomorphism must split. Then, $x$ is a retract of $f_\ast(y)$.  In fact, the object $y$ can be chosen to be a pure-injective object in $\D$ as proved in an earlier version of \cite{BW23}.     
\end{remark}
    
\begin{proposition}\label{pure closure of f_ast}
    Let $f^\ast\colon \T\to \S$ be a geometric functor between rigidly-compactly generated tt-categories, and let $\D$ be a definable subcategory of $\S$.  Then the pure-closure $\mathbf{pure}(f_\ast\D)$ of $f_\ast\D$ is a definable subcategory, where $f_\ast\D$ denotes the essential image of $\D$ under $f_\ast$. In symbols, we have that 
    \[
    \mathbf{pure}(f_\ast\D)=\mathbf{Def}(f_\ast\D). 
    \]
\end{proposition}

\begin{proof}
    This is a particular case of Proposition 4.7 (2) in \cite{BW23} applied to $f_\ast$. 
\end{proof}

\begin{remark}
The previous proposition tells us that $f_\ast\D$ may fail to be a definable subcategory of $\T$ but only because of the possible lack of closure under pure subobjects.     
\end{remark}

\begin{recollection}\label{def tensor ideals}
    Let $\T$ be a rigidly-compactly generated tt-category. A \textit{definable tensor-ideal} $\D$ of $\T$ is a definable subcategory $\D$ which is also a \textit{tensor-ideal}, that is, $\T\otimes \D\subseteq \D$. Note that $\D$ is not required  to be triangulated. In fact, a definable subcategory $\D$ satisfying that 
    \[
    x\otimes \D\subseteq \D \quad \mbox{for all} \quad x\in \T^c
    \]
is already a definable tensor-ideal (see \cite[Lemma 5.1.11]{Wag23}). Moreover, given a class of objects $\mathcal{X}$ of $\D$ which is stable under the tensor product, $\T\otimes \mathcal{X}\subseteq \mathcal{X}$, we have that $\mathbf{Def}(\mathcal{X})$ is a definable tensor-ideal (see \cite[Proposition 5.1.13]{Wag23}).
\end{recollection}

\begin{lemma}\label{lemma-pureclosure}
    Let $\T$ be a rigidly-compactly generated tt-category, and let $\mathcal X$ be a collection of objects in $\T$. If $\mathcal{X}$ is stable under the tensor product, then so is $\mathbf{pure}(\mathcal{X})$. 
\end{lemma}

\begin{proof}
    Let $x$ be an object in $\T$. By \cite[Proposition 2.10]{BKS_fields}, the functor $x\otimes -$ preserves phantom maps, and hence preserves pure monomorphisms. Indeed, this follows from the fact that the module category $\MMod\T^c$ is symmetric monoidal, and the restricted Yoneda embedding is symmetric monoidal. That is, if $g$ is a phantom map, then 
  \[\h(x\otimes g)\simeq \h(\mathrm{id}_x)\otimes \h(g)\simeq 0.\]
  Hence  $x\otimes g$ is phantom as well. The claim about pure monomorphisms follows from the relation between pure monomorphisms and phantom maps as described in  \cref{rec-phantoms-pure monos} and the fact that $x\otimes-$ preserves distinguished triangles.
\end{proof}

The following result already appears in \cite[Lemma 6.11,v2]{BW23} (cf. \cite[Lemma 3.8]{bird2026definable}) under more general assumptions. We include a proof for completeness.

\begin{lemma}\label{lemma: image of f_ast}
    Let $f^\ast\colon \mathcal{T}\to \mathcal{S}$ be a geometric functor between rigidly-compactly generated tt-categories. Let $\mathcal{D}$ be a definable tensor-ideal of $\mathcal{S}$. Then the pure-closure $\mathbf{pure}(f_\ast\mathcal{D})$ of $f_\ast\mathcal{D}$ is a definable tensor-ideal of $\mathcal{T}$.
\end{lemma}

    \begin{proof}
  We already know that $\mathbf{pure}(f_\ast\D)$ is a definable subcategory of $\T$ by  \cref{pure closure of f_ast}. It remains to verify that $\mathbf{pure}(f_\ast\D)$ is stable under the tensor product. By  \cref{lemma-pureclosure} is enough to verify that  $f_\ast\D$ is tensor-ideal. 
But this follows  by the projection formula: for any $x\in \T$ and $w\in \S$, we have  
  \[
  x \otimes f_\ast(w)\simeq f_*(f^*(x)\otimes w). 
  \]
  From here, we deduce that, for any $x\in \T$,  $x\otimes f_\ast\D$ must be in $f_\ast\D$ using that $\D$ is closed under the tensor product. It follows that $x\otimes f_\ast(w)$ is in $\mathbf{pure}(f_\ast\D)$. 
\end{proof}

\section{Conservative geometric functors}

In this section, we introduce the notion of a pure descendable family of geometric functors. An immediate consequence of this definition is that such families are jointly conservative. Let us fix some notation first. 

\begin{notation}
    Let $\mathcal{X}$ be a collection of objects of a rigidly-compactly generated tt-category $\T$. We write $\mathbf{prod}(\mathcal{X})$ to denote the closure of $\mathcal{X}$ under products, and write $\mathbf{pure}^\Delta_\Pi(\mathcal{X})$ as short for $\mathbf{pure}(\mathrm{Thick}\left(\mathbf{prod}(\mathcal{X})\right))$, that is, the pure-closure of the smallest thick subcategory of $\T$ containing the product-closure  of  $\mathcal{X}$. 
\end{notation}

%%%%%%%%-----------------------------------

\begin{definition}\label{def-puredesc}
    Let $(f_i^\ast\colon \T\to \S_i)_{i\in \mathcal I}$ be a family of geometric functors. We say that the family $(f_i^\ast)_{i\in \mathcal{I}}$ is \textit{pure descendable} if 
    \[
    \mathbf{pure}^\Delta_\Pi\left(\bigcup_{i\in \mathcal I}(f_i)_\ast\S_i\right)= \T.
    \]  
   Equivalently,  $\mathbf{pure}^\Delta_\Pi\left(\bigcup_{i\in \mathcal I}(f_i)_\ast\S_i\right)$ is a tensor-ideal and  contains the monoidal unit.
\end{definition}

Let us give sufficient conditions to ensure that $\mathbf{pure}^\Delta_\Pi\left(\bigcup_{i\in \mathcal I}(f_i)_\ast\S_i\right)$ is is a tensor-ideal.

\begin{proposition}\label{prop:pure_Pi is tensor ideal}
    Let $(f_i^\ast\colon \T\to \S_i)_{i\in \mathcal I}$ be a family of geometric functors. If $\mathbf{pure}^\Delta_\Pi(\bigcup_{i\in \mathcal I}(f_i)_\ast\S_i)$ is a definable subcategory of $\T$, then it is also a tensor-ideal. 
\end{proposition} 

\begin{proof}
    Following the proof of  \cref{lemma: image of f_ast}, we know that tensoring with an object preserves pure monomorphisms. Hence we only need to verify that the collection $\mathrm{Thick}\left(\mathbf{prod}\left(\bigcup_{i\in \mathcal I}(f_i)_\ast\S_i\right)\right)$ is stable under tensoring with arbritary objects in $\T$. In fact, by  \cref{def tensor ideals}, it is enough to show that the collection is stable under the tensor product with compact objects. But this follows since the functor $x\otimes -$, for $x\in \T^c$, commutes with products, combined with the projection formula for each geometric functor $f_i^\ast$, just as in the proof of  \cref{lemma: image of f_ast}. 
\end{proof}

\begin{corollary}\label{cor-pure=def}
If $\mathcal I$ is finite, then $\mathbf{pure}^\Delta_\Pi(\bigcup_{i\in \mathcal I}(f_i)_\ast\S_i)$ is a tensor-ideal.
\end{corollary}

\begin{proof}
    Without loss of generality we can assume that we have a single geometric functor $f^\ast\colon\T\to \S$. Since $f_\ast$ commutes with products, we  have that 
 \[
   \mathbf{pure}^\Delta_\Pi\left(f_\ast\S\right)= \mathbf{pure}\left(\mathrm{Thick}\left( f_\ast\S\right)\right). 
    \]
    By the projection formula (see proof of  \cref{lemma: image of f_ast}), we deduce that $\mathrm{Thick}\left( f_\ast\S\right)$ is tensor-ideal, and hence 
     the result follows by  \cref{lemma-pureclosure}.
\end{proof}

\begin{proposition}\label{prop:pure-desendable}
    If $(f_i^\ast\colon \T\to \S_i)_{i\in \mathcal I}$ is a pure descendable family of geometric functors,  then $(f_i^\ast)_{i \in \mathcal{I}}$ is jointly conservative. 
\end{proposition}

\begin{proof}
  By definition we have that
   \[
   \mathbf{pure}^\Delta_\Pi\left(\bigcup_{i\in \mathcal I}(f_i)_\ast\S_i\right)=\T.
   \]
  In particular, any indecomposable pure-injective object $x$ of $\T$ is a pure subobject of some object in $\mathrm{Thick}\left(\mathbf{prod}(\bigcup_{i \in \mathcal I} (f_i)_\ast \S_i)\right)$. That is, there exists a pure monomorphism $x \to y$ where $y\in \mathrm{Thick}\left(\mathbf{prod}(\bigcup_{i \in \mathcal I} (f_i)_\ast \S_i)\right)$. Since $x$ is pure-injective, this morphism splits and $x$ is a retract of $y$. In particular, $x\in \mathrm{Thick}\left(\mathbf{prod}(\bigcup_{i \in \mathcal I} (f_i)_\ast \S_i)\right)$.

  Now, since  
  \[
  \bigcup_{i \in \mathcal I} (f_i)_\ast \S_i\subseteq \left(\bigcap_{i \in \mathcal I} \ker f_i^\ast\right)^\perp
  \]
  and the latter is a colocalizing ideal, we get 
   \[
   \mathrm{Thick}\left(\mathbf{prod}\left(\bigcup_{i \in \mathcal I} (f_i)_\ast \S_i\right)\right)\subseteq \left(\bigcap_{i \in \mathcal I} \ker f_i^\ast\right)^\perp.
   \]
  In particular, any indecomposable pure-injective object $x$ of $\T$ is in $\left(\bigcap_{i \in \mathcal I} \ker f_i^\ast\right)^\perp$. By  \cref{lemma-krause 1.10}, we finally conclude that
  \[
  \bigcap_{i\in \mathcal I}\mathrm{ker}f_i^\ast=0
  \]
   as desired. 
\end{proof}

\begin{remark}
    An interesting open problem is whether the techniques of pure descendability can also be used to establish the joint conservativity of the family $(f^!_i)_{i\in \mathcal{I}}$.
\end{remark}

\begin{remark}
    One might, in fact, attempt to generalize the definition of a pure descendable family by considering the pure-closure of $\mathrm{Colocid}(\bigcup_{i\in \mathcal I}(f_i)_\ast\S_i)$, instead of  the  pure-closure of $\mathrm{Thick}\left(\mathbf{prod}(\bigcup_{i\in \mathcal I}(f_i)_\ast\S_i)\right)$.
    Note that we have an inclusion 
    \[
    \mathbf{pure}^\Delta_\Pi\left(\bigcup_{i\in \mathcal I}(f_i)_\ast\S_i\right)\subset \mathbf{pure}\left(\mathrm{Colocid}\left(\bigcup_{i\in \mathcal I}(f_i)_\ast\S_i\right)\right)
    \]
    It is straightforward to verify that the same arguments as in the proof of  \cref{prop:pure-desendable} still yield the joint conservativity of the family of functors. However, the difficulty with this alternative definition is that we do not know how to verify in practice that 
    \[
\mathbf{pure}\left(\mathrm{Colocid}\left(\bigcup_{i\in \mathcal I}(f_i)_\ast\S_i\right)\right)
    \]
    is a tensor-ideal, even when $\mathcal I$ is finite. 
\end{remark}

\begin{remark}\label{rec-weakly}
  Our motivation for introducing \textit{pure descendability} is to formulate a more flexible version of \textit{weak descendability}, one that still guarantees joint conservativity but also allows for arbitrary products. More precisely, recall that a family of geometric functors $(f_i^*\colon \mathcal{T} \to \mathcal{S}_i)_{i\in \mathcal{I}}$ is weakly descendable if 
  \[
  \mathbb{1}_\mathcal{T}\in \mathrm{Locid}\left( (f_i)_*(\mathbb{1}_{\mathcal{S}_i})\mid i\in \mathcal{I}\right).
  \]
  We refer to \cite{barthel2024homological} and \cite{Gom25} for further details on weak descendability. The key difference is that pure descendability dispenses with closure under arbitrary fibers and cofibers, while permitting products.
\end{remark}

\begin{remark}
  The precise relationship between pure and weak descendability remains unclear. We expect pure-descendability to be easier to verify in certain contexts than weak descendability. At least in the following situation we can say more: assume $\mathrm{Loc}(f_i^\ast\T)=\S_i$ for each $i\in\mathcal{I}$. In this case 
  \[
  \mathrm{Locid}((f_i)_*(\mathbb{1}_{\mathcal S_i})\mid i\in \mathcal I)=\mathrm{Locid}((f_i)_*(\mathcal{S}_i)\mid i\in \mathcal I).
  \] 
  Hence, we have the following implications: 
 \begin{enumerate}
     \item If $\T=\mathrm{Thick}(\bigcup_{i\in \mathcal{I}}(f_i)_*(\mathcal{S}_i))$, then $(f_i^\ast)_{i\in \mathcal I}$ is a family of weakly descendable functors. Indeed, we have 
  \[\mathbf{pure}^\Delta_\Pi\left(\bigcup_{i\in \mathcal I}(f_i)_\ast\S_i\right)=\mathrm{Thick}\left(\bigcup_{i\in \mathcal I}(f_i)_\ast\S_i\right)\subseteq \mathrm{Locid}\left(\bigcup_{i\in \mathcal I}(f_i)_\ast\S_i\right).\]
  In other words, under these assumptions we get that pure descendability implies weak descendability. 
  \item Suppose that $\mathcal{I}$ is finite, and each functor $f^\ast_i$ is strongly closed
  for $i \in \mathcal{I}$, that is, $(f_i)_\ast$ preserves compact objects. 
  If $\mathrm{Loc}((f_i)_*(\mathcal{S}_i)\mid i\in \mathcal I)=\T$, then  $(f_i^\ast)_{i\in \mathcal I}$ is a family of pure descendable functors. Indeed, in this case we get  
  \[
  \mathbb{1}_\T\in \mathrm{Thick}((f_i)_*(\mathbb{1}_{\mathcal{S}_i})\mid i\in \mathcal I)\subseteq \mathrm{Thick}((f_i)_*({\mathcal S_i})\mid i\in \mathcal I)
  \]
  and the latter is a $\otimes$-ideal by our assumption on $\mathcal{I}$ and the projection formula for geometric functors. In this case, we even have Grothendieck-Neeman duality, and it immediately follows that the family $(f^{!}_i)_{i\in \mathcal{I}}$ is jointly conservative as well; see \cite[Equations (3.4) and (3.9)]{BDS16}.
 \end{enumerate} 
\end{remark}

%%-----------------

\section{Pure monomorphisms for sequential limits of ring spectra}

In this section, we explain how to approach pure descendability for a family of geometric functors arising from a sequential limit of ring spectra. In particular, our goal is to describe a method for showing that the monoidal unit lies in the pure closure of the union of the images of the right adjoints in the family. To this end, we first give a general description of the setting of interest.

\begin{remark}\label{rem:strategyvanishinglimits}
    Let $(f^*_i\colon \mathcal T\to \mathcal{S}_i)_{i\in \mathcal{I}}$ be a family of geometric functors between rigidly-compactly generated tt-categories. Suppose that $\mathbb{1}_\T$ is described as a limit, $\mathbb{1}_\T\simeq \lim_\mathcal{J} t_i$ where $t_i\in (f_i)_*(\mathcal{S}_i)$ for all $i\in \mathcal{J}$.
    When is the morphism $\mathbb{1}_\T\to \prod_{\mathcal{J}}t_i$ a pure monomorphism? Given $c\in \mathcal{T}^c$, there is an equivalence 
    \[
    \Hom_\T(c,\mathbb{1}_\T)\simeq \lim_\mathcal{J} \Hom_\T(c,t_i).
    \]
    To compute the homotopy groups of the mapping space, there is a spectral sequence converging to $\Hom_\T^*(c,\mathbb{1}_\T)$ with $E_2$-term given by $\lim^j \Hom_\T^*(c,t_i)$. The strategy is to find conditions under which $\lim^j \Hom_\T^*(c,t_i)=0$ for $j>0$, since then 
    \[\Hom^*(c,\mathbb{1}_\T)\cong \lim_\mathcal{J} \Hom_\T^*(c,t_i)\hookrightarrow \prod \Hom_\T^*(c,t_i).\] 
    One can conclude that 
    \[
    \mathbb 1_\T\in \mathbf{pure}_\Pi\left(\bigcup_{i\in \mathcal I}(f_i)_\ast\S_i\right).
    \]
\end{remark}

More concretely, let us consider the following setting.

\begin{recollection}\label{rec-setting arbitrary lim}
    Let $F\colon \mathcal I\to \mathrm{CAlg}(\Pr^L_\mathrm{st})$ be a diagram such that $F(i)$ is a rigidly-compactly generated tt-category for each $i\in \mathcal I$,  and let $\S\coloneqq \lim_{\mathcal{I}} F$. Let $\T$ be a rigidly-compactly generated tt-category and assume that there is family of geometric functors $f_i^\ast\colon \T\to F(i)$ which assembles into a cone of the functor $F$. Then the canonical comparison functor   
    $f^\ast\colon\T\to \S$
   has a right adjoint $f_\ast$. Moreover, the unit of this adjunction can be described as  the map  
   \[
   x\to \varprojlim_{\mathcal{I}} (f_i)_\ast (f_i^\ast(x)) 
   \]
   induced by the unit of $x\to (f_i)_\ast (f_i^\ast(x))$ of each adjunction $(f^\ast_i,(f_i)_\ast)$. We refer to \cite[Theorem B]{HY17} for further details. 
\end{recollection}

\begin{example}\label{ex-IndSdual and S}
    In the context of  \cref{rec-setting arbitrary lim}, a natural candidate to consider is $\T=\mathrm{Ind}(\S^{\mathrm{dual}})$, where $\S^{\mathrm{dual}}$ denotes the full subcategory of $\S$ on dualizable objects. 
\end{example}

We need to recall some terminology. 

\begin{recollection}\label{rec:End-finite}
   Let $\mathcal{K}$ be a essentially small tt-category. Following \cite{Lau}, we say that $\mathcal{K}$ is End-finite if for each $x\in \mathcal{K}$, $\mathrm{End}^*_\mathcal{K}(x)$ is Noetherian. 
   
   In fact, $\mathcal{K}$ is End-finite if and only if $R=\mathrm{End}^*_\mathcal{K}(\mathbb{1})$ is Noetherian and  $\mathrm{End}^*_\mathcal{K}(x)$ is a finitely generated $R$-module for all $x\in \mathcal{K}$. This implies that $\mathrm{Hom}^*_\mathcal{K}(x,y)$ is also a finitely generated $R$-module for all $x,y\in \mathcal{K}$ since $\mathrm{Hom}^*_\mathcal{K}(x,y)$ is a direct summand of $\mathrm{End}^*_\mathcal{K}(x\oplus y)$. 
   
   In general, a rigidly-compactly generated tt-category $\mathcal{T}$ is Noetherian (or End-finite) if $R=\mathrm{End}_\mathcal{T}^*(\mathbb{1})$ is graded Noetherian and the module $\mathrm{End}_\mathcal{T}^*(c)$ is a finitely generated $R$-module for any compact object $c$ in $\T$. In other words, $\T$ is End-finite if $\T^c$ is End-finite. 
\end{recollection}

\begin{recollection}\label{rec:Monogenic end-finite}
Let $\T$ be a rigidly-compactly generated tt-category such that  $\T=\mathrm{Loc}(\mathbb 1)$. By \cite{Lau} we get that if $\mathrm{End}^\ast_\T(\mathbb{1})$ is Noetherian, then  $\T$ is End-finite. 
\end{recollection}

\begin{proposition}\label{prop:pure-mono in Artinian}
Let $k$ be commutative Artinian ring. Let $\T$ be a rigidly-compactly generated $k$--linear tt-category. Let  $F\colon \mathbb N \to \mathrm{CAlg}(\Pr^L_\mathrm{st})$ be a diagram of rigidly-compactly generated $k$--linear tt-categories such that $F(n)$ is End-finite and $\mathrm{End}^\ast_\T(\mathbb 1)$ is a finitely generated $k$-algebra, for all $n\in \mathbb N$. 

Assume that there is a family of geometric functors $f^\ast_i\colon \T\to F(i)$ which are compatible with the transition maps. Consider the induced adjunction $f^\ast\colon \T\leftrightarrows \lim_{\mathbb N} F\colon f_\ast$ as in  \cref{rec-setting arbitrary lim}.  Then the canonical map 
\[
\mathbb 1_\T \xrightarrow[]{\varphi} \prod_{n\in \mathbb N} (f_n)_\ast(\mathbb 1_n) 
\]
is a pure monomorphism provided that the unit $\mathbb{1}_\T\to f_\ast f^\ast\mathbb{1}_\T$ is an equivalence. 
\end{proposition}

\begin{proof}
First, by the description of the unit of the adjunction $(f^\ast,f_\ast)$, we obtain that the monoidal unit $\mathbb 1_\T$ is a sequential limit of the right adjoints $(f_n)_\ast$, which in turn allows to describe it  as the fiber sequence  
\begin{equation}\label{eq- 1 as a fiber}
    \mathbb{1}_\T \xrightarrow{\varphi} \prod_{n\in \mathbb N} (f_n)_\ast(\mathbb 1_n)  \xrightarrow[]{(\mathrm{id}-f_{n,n+1}^\ast)} \prod_{n\in \mathbb N} (f_n)_\ast(\mathbb 1_n) 
\end{equation}
where $f_{n,n+1}^\ast$ denote the transition functors of the sequential diagram.

Now, in order to check that $\varphi$ is a pure monomorphism, we need to verify that for any compact  object  $x \in \T$, the induced morphisms in homotopy groups
\[\pi_m\mathrm{Hom}_\T(x,\mathbb{1}_\T)\xrightarrow[]{\varphi_\ast} \pi_m\mathrm{Hom}_\T\left(x,\prod_{n\in \mathbb{N}} (f_n)_\ast(\mathbb 1_n) \right)\] are injective for all $m\geq 0$.  Since the functor $\mathrm{Hom}_\T(x,-)$ preserves fiber sequences and products, we can apply it to the fiber sequence (\cref{eq- 1 as a fiber}) to obtain $\mathrm{Hom}_\T(x, \mathbb{1}_\T)$ in a fiber sequence
\begin{center}
    \begin{tikzcd}
\mathrm{Hom}_\T(x, \mathbb{1}_\T) \arrow[r,"\varphi_\ast"]
&  \prod_{n\in \mathbb{N}} \mathrm{Hom}_\T(x,(f_n)_\ast(\mathbb 1_n)) \arrow[d, phantom, ""{coordinate, name=Z}]
\arrow[d,
rounded corners,"(\mathrm{id}-i^\ast)_\ast",
to path={ -- ([xshift=4ex]\tikztostart.east)
 |- (Z) [near start]\tikztonodes
-|([xshift=-4ex]\tikztotarget.west)
-- (\tikztotarget)}]
 \\
&  \prod_{n\in \mathbb{N}} \mathrm{Hom}_\T(x,(f_n)_\ast(\mathbb 1_n)).
\end{tikzcd}
\end{center}
where $i^\ast$ denotes the morphism induced by the $f_{n,n+1}^\ast$. Hence, from the long exact sequence induced on homotopy groups by the previous fiber sequence, we reduce our problem to verify that the cokernel of the map \[\prod_{n\in \mathbb{N}}\pi_m \mathrm{Hom}_\T(x,(f_n)_\ast(\mathbb 1_n))\xrightarrow[]{(\mathrm{id}-i^\ast)_\ast} \prod_{n\in \mathbb{N}}\pi_m\mathrm{Hom}_\T(x,(f_n)_\ast(\mathbb 1_n))\]
is trivial, for all $m\in \mathbb Z$. But such cokernel corresponds to the computation of the first higher limit $\lim^1$ for the tower of abelian groups \[\ldots\rightarrow \pi_m\mathrm{Hom}_\T(x,(f_2)_\ast(\mathbb 1_2)))\xrightarrow[]{f_1} \pi_m\mathrm{Hom}_\T(x,(f_1)_\ast(\mathbb 1_1)))\] 
where the morphisms $f_n$ are induced by the transition maps $F(n)\to F(n+1)$. We will conclude by showing that this is a tower of finitely generated $k$--modules and $k$--linear maps: then it satisfies the Mittag-Leffler condition, which implies that $\lim^1$ vanishes  under these finiteness conditions \cite[Proposition 13.2.2]{Grothendieck_EGAIII}. 

By assumption, the categories $F(n)$ are End-finite and $f^\ast_n(x)$ is compact, then  
 \[
 \mathrm{Hom}^\ast_\T(x,(f_n)_\ast(\mathbb 1_n))\simeq \mathrm{Hom}^\ast_{F(n)}(f^\ast_nx,\mathbb 1_n)
 \]
 is a finitely generated $\mathrm{End}^\ast(\mathbb{1}_n)$-module; in particular, it is a finitely generated $k$--module in each degree since  $\mathrm{End}^\ast(\mathbb{1}_n)$ is Noetherian. 
\end{proof}

We now apply \cref{prop:pure-mono in Artinian} to categories of the form $\Mod_R$, where $R$ is a commutative ring spectrum. Such categories carry a canonical action of $\pi_\ast(R)$. In particular, if $\pi_\ast R$ is a $k$-algebra, then $\Mod_R$ acquires the structure of a $k$-linear category.

\begin{corollary}\label{coro-puremono-seqlimit}
   Let $k$ be a commutative Artinian ring. Let $R$ be a commutative ring spectrum such that $\pi_\ast R$ is a $k$-algebra. Suppose that 
   \[
   R\simeq \lim_{n\in \mathbb N } R_n
   \]
   in $\mathrm{CAlg}$, with $\pi_\ast R_n$ a finitely generated $k$-algebra. Then the canonical map
   \[
   R\to \prod_{n\in \mathbb N} R_n 
   \]
   is a pure monomorphism in $\Mod_R$. 
\end{corollary}

\begin{proof}
    First, we show that the categories $\Mod_{R_n}$ are End-finite. Since $k$ is Artinian, it is in particular Noetherian, and therefore $\pi_\ast R_n$ is Noetherian as well. The claim follows, and the result now is a consequence of \cref{prop:pure-mono in Artinian}.
\end{proof}

\begin{remark}\label{rem-othercond for ML}
  We note that most of the proof of \cref{prop:pure-mono in Artinian} is formal, and that the assumption that the categories $\T$ and $F(n)$ are $k$-linear where $k$ is a commutative Artinian ring is only used to apply the Mittag--Leffler condition for the vanishing of higher limits. Of course, one could instead impose other hypotheses that ensure this condition. For instance, it would suffice to require that $\T$ is $k$-linear for an arbitrary commutative ring and that $\Hom_{F(n)}(c,\mathbb 1_n)$ is an Artinian $k$-module for every compact object $c \in F(n)$. As the following example illustrates, this situation does indeed arise.  
\end{remark}

We thank Isaac Bird for the following observation:

\begin{example}\label{ex-completelocal}
 Let $R$ be a complete Noetherian local commutative ring with maximal ideal $\mathfrak{m}$. Then $R/\mathfrak{m}^n$ is a local commutative Artinian ring for all $n$. Since it is also Noetherian, we have that both the categories $\Mod_R$ and $\Mod_{R/\mathfrak{m}^n}$ are End-finite. In particular, 
 \[
 \Hom_{R}^*(c,(f_n)_*(R/\mathfrak{m}^n))=\Hom_{R/\mathfrak{m}^n}^*(f_n^*(c),R/\mathfrak{m}^n)
 \]
 are degree-wise Artinian $R/\mathfrak{m}^n$-modules, and hence Artinian $R$-modules via restriction along the surjective map $R\to R/\mathfrak{m}^n$. Since the images of transition maps from $m\geq n$ are sub-$R$-modules, we have that they stabilize and the Mittag-Leffler condition holds. 
  Now, since $R$ is complete, we have 
    \[
    R\simeq \lim_{n\in \mathbb{N}} R/\mathfrak{m}^n .
    \]
   Viewing the rings $R$ and $R/\mathfrak{m}^n$ as commutative ring spectra via Eilenberg--MacLane spectra, we are in the setting of \cref{coro-puremono-seqlimit}, except that $R$ is not necessarily a $k$-algebra for any obvious choice of Artinian ring $k$. However, the discussion above shows that we may instead apply \cref{rem-othercond for ML}. Consequently, the canonical map
    \begin{equation}\label{eq-completelocal}
        R\to \prod_n R/\mathfrak{m}^n
    \end{equation} 
    is a pure monomorphism. Furthermore, by \cite[Theorem 3]{Warfield69}, the ring $R$ is a pure injective object in the category of discrete $R$–modules, and hence also pure injective in $\Mod_R$. It follows that the map in \eqref{eq-completelocal} splits. Therefore, $R$ is a direct summand of $\prod_{n} R/\mathfrak m^{n}$. Similar phenomena have been studied in \cite{Lar03}.
\end{example}

\begin{remark}
   We have provided sufficient conditions to ensure that the monoidal unit $R$ lies in $  \mathbf{pure}^\Delta_\Pi\left(\bigcup_{i\in \mathcal I}(f_n)_\ast\Mod_{R_n}\right)$. The next natural step is to establish conditions under which this subcategory is also closed under tensor products. Such a result would imply that the family of geometric functors $(f_n^\ast)_{n\in \mathbb{N}}$ is pure descendable, and hence jointly conservative by   \cref{prop:pure-desendable}. We address this question in the following section in the setting of cochain algebras of an Artinian locally finite group. In contrast, we do not know whether the corresponding property should be expected to hold in the context of  \cref{ex-completelocal}.  
\end{remark}

\section{Pure descendability for cochains on classifying spaces}\label{sec:puremono_classifying spaces}

In this section, we investigate the following situation. Let $G$ be a group, and $\mathcal F$ a family of subgroups. The goal of this section is to study the family of functors $(\textrm{Ind}^H_G\colon \mathrm{Mod}_{ C^*(BG;R)}\rightarrow \mathrm{Mod}_{C^*(BH;R)})_{H\in \mathcal{F}}$ and find conditions under which the morphism $C^*(BG;R)\to \prod_{H\in \mathcal{F}} {C^*(BH;R)}$ in $\mathrm{Mod}_{ C^*(BG;R)}$ is a pure monomorphism. We focus on a particular situation: locally finite Artinian groups.

\begin{notation}
    Let $G$ be a group and $R$ a commutative ring spectrum. We write $BG$ to denote the classifying space of $G$, and $C^\ast(BG;R)=F(\Sigma^\infty_+BG,R)$ to denote the function spectrum of $R$--valued cochains on $BG$, and $\mathrm{Mod}_{C^\ast(BG;R)}$ to denote the category of module spectra over the ring spectrum $C^\ast(BG;R)$. If $R=Hk$ where $k$ is a field of positive characteristic $p$, we simply write $C^\ast(BG;k)$.
    
    For a subgroup $H \leq G$, then the induced ring morphism $C^\ast(BG;R)\to C^\ast(BH;R)$ gives a structure of $C^\ast(BG;R)$-module to $C^\ast(BH;R)$. 
    
    We denote by $\textrm{Res}^H_G\colon \mathrm{Mod}_{C^\ast(BH;R)} \to \mathrm{Mod}_{C^\ast(BG;R)}$ the restriction functor. This functor admits a left adjoint $\textrm{Ind}^H_G\colon \mathrm{Mod}_{ C^*(BG;R)}\rightarrow \mathrm{Mod}_{C^*(BH;R)}$, referred to as induction functor. 
\end{notation} 

\begin{warning}
   Note that our notation for the induction and restriction functors is intended to indicate the direction of the functors, in accordance with extension and restriction of scalars along the ring map $C^\ast(BG;R) \to C^\ast(BH;R)$. This convention differs from that used in other contexts, such as the representation theory of groups, where one writes $\textrm{Ind}_H^G$ for the induction functor from a subgroup $H$ to an ambient group $G$, and similarly for the restriction functor. 
\end{warning}

\begin{notation}
   For simplicity, we write $\textrm{Hom}_G(-,-)\coloneqq \textrm{Hom}_{\Mod_{C^\ast(BG;R)}}(-,-)$.
\end{notation}

\begin{remark}\label{rem-cohomology is k-algebra}
   Let $G$ be a group and $k$ a commutative ring. In this case, we have  $\pi_{-i}(C^\ast(BG;k))\cong H^i(G;k)$, hence $\pi_\ast C^\ast(BG;k)$ is naturally a  $k$-algebra. Since $\Mod_{C^\ast(BG;k)}$ is $\pi_\ast C^\ast(BG;k)$-linear, it follows in particular that it is $k$-linear.
\end{remark}

We now present the first example of a pure monomorphism in this context.

\begin{example}
     Let $k$ be a field of characteristic $p$, and let $G$ be a finite group and $S$ be a Sylow $p$-subgroup of $G$. Then $C^\ast(BG;k)\rightarrow C^*(BS;k)$ is a pure monomorphism. Indeed, this follows using the transfer which satisfies Frobenius reciprocity. 
\end{example}

\begin{recollection}\label{rec: tree}
Let $G$ be locally finite countable group expressed as $G=\cup_{n\in \mathbb{N}} G_n$ for an ascending chain of finite subgroups $\cdots \subset G_n\subset G_{n+1}\subset \cdots$ of $G$. Then $G$ acts simplicially on a tree $T$ with finite isotropy groups; in fact, the isotropy subgroups for this action correspond precisely to the subgroups $G_n$. Moreover, the fundamental domain for the action can be identified with a ray  (e.g., see \cite[Section 5]{gomez2024picard}). This gives rise to a graph of finite groups which can depicted as follows: 

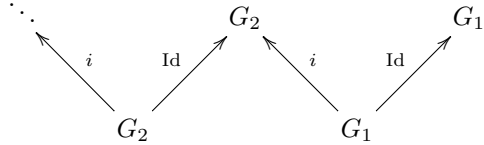
\centerline{\xymatrix{ 
\ddots \ar@{<-}[rd]^{i}  & &  G_2 \ar@{<-}[rd]^{i} \ar@{<-}[ld]_{\textrm{Id}}  & & G_1 \ar@{<-}[ld]_{\textrm{Id}}   \\ & G_2   & &  G_1  &  
}}

\noindent where $i$ denotes the inclusions $G_n\hookrightarrow G_{n+1}$. In fact, $G$ corresponds to the fundamental group of the above graph of groups in the terminology of Bass-Serre theory.   
\end{recollection}

\begin{recollection}
    \label{rec: structure l.f. Artinian group}
    Let $G$ be locally finite Artinian group (e.g. a discrete $p$-toral group). Following the arguments in the proof of \cite[Proposition 1.2]{BLO07}, $G$ fits into an extension 
    \begin{equation}
        1\to A \to G \to \pi \to 1
    \end{equation}
    where $\pi$ is a finite group and $A$ is a finite product of groups of the form $\mathbb{Z}/q^n$ with $q$ a prime number and $1\leq n\leq \infty$. Note that $A$ is countable and therefore locally finite Artinian groups are also countable and can be expressed as $G=\cup_{n\in \mathbb{N}} G_n$ with $G_n$ finite groups. 
    
    Moreover, they satisfy finiteness conditions with respect to conjugacy classes of certain subgroups. Given $N>0$ there is a finite set of conjugacy classes of subgroups of order $N$ and they are of finite $p$-rank for every prime $p$, (see the proof of \cite[Lemma 1.4]{BLO07}). In particular, there are finitely many conjugacy classes of elementary abelian subgroups.
\end{recollection}

\begin{proposition}\label{decomposition of C*BS}
   Let $R$ be a commutative ring spectrum.  Let $G$ be a countable locally finite group acting on a tree $T$ as in  \cref{rec: tree}. 
    Then $C^\ast(BG;R)$ agrees with the homotopy equalizer 
 \begin{equation}\label{eq: eq}
     C^\ast(BG;R)\stackrel{\simeq}{\rightarrow}\mathrm{eq}\left( \prod_{n\in \mathbb{N}}C^\ast(BG_n;R)\rightrightarrows \prod_{n\in \mathbb{N}}C^\ast(BG_n;R)\right)
 \end{equation}
 in the category $\Mod_{C^\ast(BG;R)}$ where the arrows are induced by identity maps $G_n\to G_n$ and from the inclusion $G_n \hookrightarrow  G_{n+1}$, where we view $C^\ast(BG_n;R) \in \Mod_{C^\ast(BG;R)}$ via ${C^\ast(BG;R)}\to C^*(BG_n;R)$. Equivalently, $C^\ast(BG;R)$ is the sequential limit of 
 \[
 \cdots \to C^*(BG_{n+1};R) \rightarrow C^*(BG_{n};R) \to \cdots\ldots\to C^*(BG_2;R)\to C^*(BG_1;R). 
 \]
\end{proposition}

\begin{proof}
The classifying space $BG$ is homotopy equivalent to the homotopy coequalizer 
\[
\mathrm{coeq}\left(\coprod_{n\in \mathbb{N}} BG_n \rightrightarrows \coprod_{n\in \mathbb{N}} BG_n \right)
\]
where one of the arrows is the identity and the other one is induced by the inclusions $G_n\hookrightarrow  G_{n+1}$.  This will follow by the existence of the tree $T$. Since $T$ is contractible, $BG\simeq (T)_{hG}$ which in turn can be described as a homotopy colimit. In fact, it is the telescope of $\cdots \to BG_n \rightarrow BG_{n+1} \to \cdots$; which in turn can be expressed as a coequalizer.

 The statement follows from the fact that $C^\ast(-;R)$ commutes with homotopy colimits. Moreover, all maps are $C^\ast(BG;R)$-module morphisms via the subgroup inclusions $G_n\subseteq G_{n+1}\subseteq G$; that is all arrows live in the category $\Mod_{C^\ast(BG;R)}$.
\end{proof}

Recall that a essentially small tt-category $\mathcal{K}$ is \textit{End-finite} if for each $x\in \mathcal{T}$, $\mathrm{End}^*_\mathcal{K}(x)$ is Noetherian. A rigidly-compactly generated tt-category $\T$ is \textit{End-finite} if $\T^c$ is End-finite; see  \cref{rec:End-finite} for a more detailed discussion.

\begin{example}\label{ex-mod is end-finite}
    Let $G$ be a finite group, and $R$ be a commutative Noetherian ring. Then  $\Mod_{C^\ast(BG;R)}$ is End-finite. Indeed, this follows by  \cref{rec:Monogenic end-finite} since 
    \[
    \mathrm{End}_{\Mod_{C^\ast(BG;R)}}^\ast({C^\ast(BG;R)})\simeq H^\ast(BG;R)
    \]
    which is a finitely generated $R$--algebra by a classical result of Evens and Venkov.
\end{example}

\begin{proposition}\label{prop:pure mono}
  Let $G$ be locally finite Artinian group expressed as $G=\bigcup_{n\in \mathbb{N}} G_n$ for an ascending chain of finite subgroups $\cdots \subset G_n\subset G_{n+1}\subset \cdots$ of $G$, and let $R$ be a commutative Artinian ring. Then the morphism  
   \[\varphi \colon C^\ast(BG;R)\to \prod_{n} \mathrm{Res}_{G}^{G_n}(C^*(BG_n;R))\] in $\mathrm{Mod}_{C^\ast(BG;R)}$ is a pure monomorphism. 
\end{proposition}

\begin{proof} 
By  \cref{decomposition of C*BS}, we know that $C^\ast(BG;R)$ the  sequential limit of the $C^\ast(BG_n;R)$ and the obvious maps induced by inclusions of subgroups. By  \cref{ex-mod is end-finite}, each category $\Mod_{C^\ast(BG_n;R)}$ is End-finite since $\mathrm{End}^\ast(C^\ast(BG_n;R))=H^*(BG_n;R)$ is finitely generated as an $R$-algebra. Hence, the result follows from  \cref{coro-puremono-seqlimit}.
\end{proof}

\begin{recollection}\label{rec-abelem}
   Let $G$ be a finite group.  Then $C^\ast(BG;\mathbb F_p)$ lies in the thick subcategory generated by 
   \[
   \prod_{E\in \mathrm{Abelem}(G)} \mathrm{Res}^E_G C^\ast(BE;\mathbb F_p)
   \]
    where $\mathrm{Abelem}(G)$ denotes the set of elementary abelian subgroups of $G$. Indeed, this follows from the fact that the category $\Mod_{C^\ast(BG;\mathbb F_p)}$ can be identified  with the localizing subcategory generated by the monoidal unit in $\mathbf{K}(\mathrm{Inj}(\mathbb F_p G))$; see \cite[Remark 7.5]{BK08} and \cite[Theorem 3.1]{BG}. 
\end{recollection}

\begin{theorem}\label{thm-example of pure-des}
    Let $G$ be locally finite Artinian group expressed as $G=\bigcup_{n\in \mathbb{N}} G_n$ for an ascending chain of finite subgroups $\cdots \subset G_n\subset G_{n+1}\subset \cdots$ of $G$. Then the family of geometric functors 
    \[
    \left(\mathrm{Ind}_G^{G_n}\colon \Mod_{C^\ast(BG;\mathbb F_p)} \to \Mod_{C^\ast(BG_n;\mathbb F_p)} \right)_{n\in \mathbb N}
    \]
    is pure descendable. 
\end{theorem}

\begin{proof}
    By  \cref{prop:pure mono}, we have that 
    \[
    C^\ast(BG;\mathbb F_p) \in \mathbf{pure}^\Delta_\Pi\left( \bigcup_{n\in \mathbb N} \mathrm{Res}^{G_n}_G \Mod_{C^\ast(BG_n;\mathbb F_p)} \right). 
    \]
    It remains to show that the latter is a tensor ideal of $\Mod_{C^\ast(BG;\mathbb F_p)}$. For this we will use the family of elementary abelian subgroups. Let $\mathcal{E}$ be the set of conjugacy classes of elementary abelian subgroups of $G$ which is finite by  \cref{rec: structure l.f. Artinian group}, and $\mathcal{E}_n$ the set of $G_n$-conjugacy classes of elementary abelian subgroups of $G_n$. We choose representatives such that we have a surjection $q\colon \sqcup_n \mathcal{E}_n \twoheadrightarrow \mathcal{E}$ since any elementary abelian subgroup is contained in some $G_n$, $\sqcup_n \mathcal{E}_n=\sqcup_{E\in \mathcal{E}}q^{-1}(E)$. We have a commutative diagram 
    \begin{center}
    \begin{tikzcd}
        \mathrm{Mod}_{C^\ast(BG;\mathbb F_p)} \arrow[rr,"\prod \mathrm{Ind}^{G_n}_G"] \arrow[dd,"\prod \mathrm{Ind}^{E}_{G}"'] & & \prod_{n\in \mathbb N}  \mathrm{Mod}_{C^\ast(BG_n;\mathbb F_p)} \arrow[dd,"\prod \mathrm{Ind}^{V}_{G_n}"']  \\ 
         & & \\
        \prod_{E\in \mathcal{E}}   \mathrm{Mod}_{C^\ast(BE;\mathbb F_p)}\simeq \mathrm{Mod}_{\prod_\mathcal{E} C^\ast(BE;\mathbb F_p)}\arrow[rr," "] & & \prod_n \prod_{V\in \mathcal{E}_n}   \mathrm{Mod}_{C^\ast(BV;\mathbb F_p)} 
    \end{tikzcd}
\end{center}
where the bottom horizontal arrow is induced by conjugations: if $E\in \mathcal{E}$ and $V\in q^{-1}(E)\cap \mathcal{E}_n$, they are conjugated subgroups in $G_n$, then there are equivalences that all assemble into $f\colon C^*(BE;\mathbb F_p)\to \prod_{V\in q^{-1}(E)}C^*(BV;\mathbb F_p)$. 

By  \cref{rec-abelem}, for each $n$, $C^*(BG_n;\mathbb F_p)\in \mathrm{Thick}(\textrm{Res}^E_{G_n}(C^*(BE;\mathbb F_p))\mid E\in \mathcal{E}_n)$,
that is the thick closure of the essential image of $(\mathrm{Res}^{E}_{G_n})_{E\in \mathcal{E}_n}$. We deduce that $\mathrm{Res}^{G_n}_G C^\ast(BG_n;\mathbb F_p)$ is in the thick closure of the essential image of $(\mathrm{Res}^{E}_{G})_{E\in \mathcal{E}_n}$. 

Now, the commutativity of the diagram shows that the thick closure of the essential image of $(\mathrm{Res}^{E}_{G})_{E\in \mathcal{E}}$ agrees with the thick closure of $\bigcup_n \textrm{Res}^{G_n}_G$. Since the set $\mathcal{E}$ is finite, we can invoke  \cref{cor-pure=def} to conclude that the pure-closure $\mathbf{pure}^\Delta_\Pi\left( \bigcup_{n\in \mathbb N} \mathrm{Res}^{G_n}_G \Mod_{C^\ast(BG_n;\mathbb F_p)} \right)$ is a tensor-ideal, as desired.  
\end{proof}

\begin{corollary}\label{coro: Chouinard for discrete p-toral}
   Let $G=\cup_{n\in \mathbb N}G_n$ be locally finite Artinian group, 
   and $\mathcal{E}$ be a set of representatives of $G$-conjugacy classes of elementary abelian $p$-subgroups of $G$. Consider the functor 
    \[
    \mathrm{Ind}_{\mathcal{E}}\colon \Mod_{C^*(BG;\mathbb{F}_p)}\rightarrow \prod_{E\in\mathcal{E}} \Mod_{C^*(BE;\mathbb{F}_p)}
    \]
    induced by the functors $\mathrm{Ind}_G^E\colon \Mod_{C^*(BG;
    \mathbb{F}_p)}\rightarrow \Mod_{C^*(BE;\mathbb{F}_p)}$. Then $ \mathrm{Ind}_{\mathcal{E}}$  is conservative.  
\end{corollary}

\begin{proof}
 This follows by Chouinard's theorem for finite groups (see \cite[Theorem 3.1]{BG}) combined with  \cref{thm-example of pure-des}. 
\end{proof}

\begin{remark}
      Locally finite Artinian $p$-groups (discrete $p$-toral groups, see  \cref{rec: structure l.f. Artinian group}) are essential pieces in the theory of $p$-local compact groups as discrete models for classifying spaces of compact Lie groups and $p$-compact groups developed by Broto, Levi and Oliver \cite{BLO07}. 

      In \cite[Theorem 4.25]{BCHV} the authors prove a Chouinard's theorem for $p$-local compact groups with respect to the family of elementary abelian subgroups under the assumption that the Sylow $p$-subgroup $S$ is geometric; that is, $S$ is a discrete approximation at $p$ of a $p$-toral group. \cref{coro: Chouinard for discrete p-toral} allows one to establish this result in the case of induction with no restriction on the Sylow $p$-subgroup. 
\end{remark}

%------------ bibliography

\bibliographystyle{alpha}
\bibliography{bibfile} 

\end{document}